\documentclass[12pt]{amsart}
\usepackage{amsmath,amssymb,graphicx,mathrsfs}
\usepackage{fullpage}
\usepackage[all]{xy}

\newtheorem{theorem}{Theorem}[section]
\newtheorem{lemma}[theorem]{Lemma}

\newtheorem{proposition}[theorem]{Proposition}

\theoremstyle{remark}
\newtheorem*{example}{Example}
\newtheorem*{remark}{Remark}

\newcommand{\N}{\mathbb N}
\newcommand{\C}{\mathbb C}
\newcommand{\R}{\mathbb R}

\def\({\left(}
\def\){\right)}
\newcommand{\abs}[1]{\left| #1 \right|}
\newcommand{\cS}{\mathcal S}
\newcommand{\cO}{\mathcal O}
\newcommand{\He}{\mathcal{H}}

\DeclareMathOperator{\Hom}{Hom}

\DeclareMathOperator{\vol}{vol}
\DeclareMathOperator{\Ind}{Ind}

\DeclareMathOperator{\Xcch}{X_\bullet}
\newcommand{\GL}{\mathrm{GL}}

\newcommand{\Sp}{\mathrm{Sp}}

\newcommand{\PGSp}{\mathrm{PGSp}}
\newcommand{\Spin}{\mathrm{Spin}}
\newcommand{\GSp}{\mathrm{GSp}}
\newcommand{\SL}{\mathrm{SL}}

\newcommand{\PGL}{\mathrm{PGL}}

\newcommand{\SO}{\mathrm{SO}}

\newcommand{\smatr}[4]{\(\begin{smallmatrix} #1 & #2 \\ #3 & #4\end{smallmatrix}\)}
\newcommand{\smthree}[9]{\(\begin{smallmatrix} #1 & #2 & #3 \\ #4 & #5 & #6 \\ #7 & #8 & #9 \end{smallmatrix}\)}

\newcommand{\wt}{\widetilde}
\newcommand{\wh}{\widehat}

\def\bfG{\mathbf{G}}

\def\bfV{\mathbf{V}}
\def\Oct{\mathbb{O}}
\DeclareMathOperator{\Tr}{Tr}

\title{Matching of Hecke operators for Exceptional dual pair correspondences}

\dedicatory{
to Steve Rallis, in memoriam} 

\author{Gordan Savin and Michael Woodbury}

\address{G. S.: Department of Mathematics, University of Utah, Salt Lake City, UT}\email{savin@math.utah.edu}
\address{M. W.: Department of Mathematics, Columbia University, New York, NY} \email{woodbury@math.columbia.edu}

\thanks{Partially supported by NSF grant DMS 0852429.}

\begin{document}

\begin{abstract}  
Let $\mathbf{G}$ be a split algebrac group of type $E_n$ defined over a $p$-adic field.  This group contains a dual pair $G \times G'$ 
where one of the groups is of type $G_2$. The minimal representation of $\mathbf{G}$, when restricted to the dual pair, gives a 
correspondence of representations of the two groups in the dual pair. We prove a matching of spherical Hecke algebras of $G$ and $G'$, 
when acting on the minimal representation. This implies that the correspondence is functorial, in the sense of Arthur and Langlands, 
for spherical representations. 
\end{abstract} 

\maketitle

\section{Introduction}

Let $F$ be a $p$-adic field with ring of integers $\cO$ and fixed inverse of uniformizer $\varpi$.   Let $q$ be the order of the 
residual field.  We fix an absolute value on $F$ so that $|\varpi |=q$. 
We consider exceptional dual pairs $G\times G'$ inside of an adjoint group $\mathbf{G}$ where each group consists of the $F$ points of a split reductive algebraic group.  
 Denote by $K$ and $K'$ hyperspecial maximal compact subgroups of $G$ and $G'$ respectively.  Let $(\Pi,\bfV)$ be the minimal representation of $\bfG$.

If $\sigma'$ is an irreducible representation of $G'$, we call an irreducible representation $\sigma$ of $G$ a \emph{$\Theta$-lift of $\sigma'$} if $\sigma\otimes \sigma'$ is a quotient of $\Pi$.  If $\sigma,\sigma'$ are spherical, we will prove  that the correspondence $\sigma\leftrightarrow \sigma'$ is functorial with respect to a natural injection on the dual groups
 $$ r: \wh{G}'(\C)\to \wh{G}(\C). $$

To be more precise, let $C$ be the centralizer of $r(\widehat G')$ in $\widehat G$. Then $C$ is a reductive, possibly 
finite,  group. Let $f: \SL_2(\C) \rightarrow C$ be a map corresponding to the regular unipotent 
orbit in $C$ by Jacobson-Morozov. Let
\[ 
s= f\smatr{q^{1/2}}{0}{0}{q^{-1/2}}. 
\] 
Let $\He$ and $\He'$ denote the spherical Hecke algebras of $G$ and $G'$ respectively.  Let $T\in\He$ correspond, via the Satake isomorphism, to a finite dimensional representation $V$ of $\wh{G}(\C)$.  (We describe this in more detail in Section~\ref{sec:satake} below.)    Write $V=\sum V'\otimes V'' $, the restriction of $V$ to $\wh G' \otimes C$. 
 We define a map $\tilde{r}:\He \to \He'$ by 
\begin{equation}\label{eq:rtildemap}
 \tilde{r}(V)= \sum_{V'} \Tr_{V''}(s) V' .
\end{equation}

Note that if $C$ is finite, then $s=1$, the identity in $\wh{G}$, and $\tilde{r}(V)$ is just the  restriction of $V$ to $\wh G'$. 
We consider the following dual pairs:
\begin{table}[h]
$$
\begin{array}{l||c|c|c|c|c}
\bfG & D_4 & D_5 & E_6 & E_7 & E_8 \\ \hline
G'& S_3 & \PGL_2 & \PGL_3 &   G_2   & G_2 \\ \hline
G & G_2 &  G_2   &   G_2  & \PGSp_6 & F_4  
\end{array}
$$
\caption{Dual pairs $G\times G'\subset \bfG$}
\label{table:dualpairs}
\end{table}

\begin{theorem}\label{theorem}
For the dual pairs $G\times G'\subset \bfG$ in Table \ref{table:dualpairs}, $T\in \He_{G}$ and $\wt{r}$ given by \eqref{eq:rtildemap}, $\Pi(T)=\Pi(\wt{r}(T))$ as operators on $\bfV^{K\times K'}$.
\end{theorem}
As a matter of terminology, if the actions of $T$ and $\wt{r}(T)$ agree on a space $V$, or, more precisely, on a subset of fixed vectors, for all $T\in \He_G$ we will say there is a \emph{matching of Hecke operators} of $\He_G$ and $\He_{G'}$ on $V$ or, more concisely, \emph{matching} on $V$.  We trust that the precise space of fixed vectors will be clear from context.

An analogue of Theorem \ref{theorem} is well known in the case of classical theta correspondences \cite{mvw}. 
For exceptional groups, the first example of matching was obtained by Rallis and Soudry in \cite{rs}. 

The proof of Theorem \ref{theorem} is by induction on the rank of $\bfG$. 
The main tool is Jacquet functors of  $\bfV$ with respect to maximal parabolic subgroups of dual pairs. Most of the 
needed functors were computed in \cite{ms}. One remaining, but rather remarkable case (for $\bfG=E_8$), is computed in the last section.

\section{The Satake isomorphism}\label{sec:satake}

Before giving the proof of Theorem~\ref{theorem} we review the facts about the Hecke algebra and the Satake isomorphism (most of which can be found in \cite{gross}) which will be relevant, and we give some general lemmas which will be key in the proof of Theorem~\ref{theorem}.

\subsection{The Hecke algebra}

Let $G$ be a split reductive group, $K$ a hyperspecial maximal compact subgroup, $B=TU$ a Borel subgroup.  There is an Iwasawa decomposition $G=BK$, and the choice of Borel gives a set $\Phi^+$ of positive roots.

We identify the cocharacters of $T$, $\Xcch(T)$, with the coweight lattice $\Lambda_c$ so that for every cocharacter 
$\lambda:F^\times\to T$, the adjoint action of $\lambda(t)$ on the root subgroup of $U$, corresponding to a root $\alpha$, is 
given by the multiplication by the scalar $t^{\langle \lambda, \alpha \rangle}$. We have a  Cartan decomposition:

\begin{proposition}\label{prop:doublecosets} 
The group $G$ is the disjoint union of double cosets $K\lambda(\varpi) K$ for $\lambda\in \Lambda_c^+$ where 
$$ \Lambda_c^+ = \{ \lambda \in \Lambda_c\mid \langle \lambda, \alpha \rangle \geq 0 \mbox{ for all } \alpha\in \Phi^+\}. $$
\end{proposition}

\begin{example}
$G=\GL_3(F)$, $\wh{G}=\GL_3(\C)$.  If $\alpha_1=(1,-1,0)$ and $\alpha_2=(0,1,-1)$ are the simple roots, then 
$P^+$ consists of  $\lambda=(l,m,n)$ such that $l\geq m\geq n$. Then 
 $$ \lambda(t) = \smthree{t^l}{}{}{}{t^m}{}{}{}{t^n}$$ 
and it is easy to verify directly, from the theory of elementary divisors, that 
 $$ \GL_3(F) = \bigsqcup_{l\geq m\geq n} \GL_3(\cO) \smthree{\varpi^{l}}{}{}{}{\varpi^{m}}{}{}{}{\varpi^{n}} \GL_3(\cO). $$
\end{example}

Irreducible representations of $\wh{G}$, the complex dual group of $G$, are parameterized by their highest weights $\lambda\in \Lambda_c^+$.  Let $R(\wh{G})$ be the representation ring of $\wh{G}$.  That is, $R(\wh{G})$ is the $\C$-vector space with basis consisting of the irreducible representations of $\wh{G}$.  We denote the representation of highest weight $\lambda$ by $V_\lambda$, and consider the map 
\begin{equation}\label{eq:VtoT}
 R(\wh{G})\to \C[\Lambda_c] \qquad V_\mu \mapsto \sum_{\mu}m_\lambda(\mu)[\mu] 
\end{equation}
where $m_\lambda(\mu)$ is the dimension of the $\mu$-weight space in $V_\lambda$.  Letting $W$ denote the Weyl group, this gives an isomorphism $R(\wh{G})\simeq \C[\Lambda_c]^W$.

The Hecke algebra $\He_G$ consists of all locally constant compactly supported $K$-biinvariant functions $f:G\to \C$.  By Proposition~\ref{prop:doublecosets}, $\He_G$ has a basis consisting of the characteristic functions of $K\lambda(\varpi)K$ with $\lambda\in \Lambda_c^+$.  We denote these by $T_\lambda$.

If $(\sigma,V)$ is a smooth $G$-module then the action of $f\in \He_G$ is given by
 $$ f*v = \int_{G}f(g)\sigma(g)vdg. $$ 
We normalize the Haar measure $dg$ so that $\vol(K)=1$.  Let $U$ be the unipotent radical of $B$ as above.  By the Iwasawa decomposition, we can write
\begin{equation}\label{eq:double2single}
 K\lambda(\varpi)K = \bigcup_u\bigcup_t ut  K
\end{equation}
for some representatives $t\in T$ and $u\in U$.  So if $v$ is a $K$-fixed vector then
\begin{equation}\label{eq:heckeactionfull}
 T_\lambda* v = \sum_{u,t} \sigma(u)\sigma(t)v.
\end{equation}
If $r_U: V \rightarrow V_U$ is the natural projection, where $V_U$ is the space of $U$-coinvariants, then 
\begin{equation}\label{eq:heckeaction}
 r_U(T_\lambda* v) = \sum_t n(t) r_U(\sigma(t)v)
\end{equation}
where $n(t)$ is the number of single cosets of type $utK$ appearing in $K\lambda(\varpi)K$.

\subsection{The (relative) Satake transform}

Let $\delta_U$ denote the modular character of $B$ given by
 $$d(bub^{-1})=\delta_U(b)du.$$
The measure $du$ is normalized so that $\vol(K\cap U)=1$.  Obviously, $\delta_U$ is trivial on $U$, and so this defines a character $\delta:T\to \R^\times_+$.  We take $\delta_U^{1/2}(t)$ to be the positive square root of this character.  Let $\Phi^+$ denote the positive roots of $G$ (determined by $B$.)  Then if $\mu\in X_\bullet(T)$,
\begin{equation}\label{eq:modchar}
 \delta^{1/2}(\mu(\varpi))=q^{\langle \mu,\rho \rangle}, \qquad 2\rho=\sum_{\alpha\in \Phi^+} \alpha.
\end{equation}

The Satake transform $\cS_T:\He_G\to \He_T$ is given by
\begin{equation}\label{eq:sataki}
 \cS_Tf(t) = \delta(t)^{1/2} \int_{U} f(tu)du.
\end{equation}
It is a fact that $\cS_T$ is injective and its image is equal to the Weyl group invariants.  Since $\He_T = \C[\Xcch(T)]=\C[\Lambda_c]$, it follows from our discussion above that this defines an isomorphism $\cS:\He_G\to R(\wh{G})$.

The Satake transform can be defined analogously for any parabolic $P=MN$.  This yields the \emph{relative Satake transform} $\cS_M: \He_G\to \He_M$:
\begin{equation}\label{eq:relsataki}
 \cS_Mf(m) = \delta_N^{1/2}(m)\int_Nf(mn)dn. 
\end{equation}
As above, $dn$ is the measure which gives $N\cap K$ volume 1, and $\delta_N:M\to \R^\times_+$ is the modular character.

We may assume $P\supset B$, so the composition of $\cS_M$ with the Satake transform from $\He_M$ to $\He_T$ is $\cS_T$.

\subsection{Some general lemmas}\label{sec:generallemmas}

In this section, we prove various simple lemmas which will be used in our proof of Theorem~\ref{theorem}.

Throughout this paper parabolic induction and the Jacquet functors  will be normalized as follows.  Let $P=MN$ be a parabolic subgroup of $G$.  Suppose that $(\sigma,W)$ is a representation of $M$ which we extend trivially to $P$.  Then we define $(\rho, i_P^G(W))$ to be the representation of $G$ consisting of smooth functions $f:G\to W$ which satisfy,
\begin{equation}\label{eq:pinduction}
 f(mng) = \delta_N^{1/2}(m)\sigma(m)f(g) \qquad \mbox{ for all $m\in M,n\in N,g\in G$}
\end{equation}
with right regular action $\rho(g)f:h\mapsto f(hg)$.  Note that the usual induction functor is
\begin{equation}\label{eq:usualind}
 \Ind_P^G(W) = i_P^G( \delta_N^{-1/2}\otimes W ).
\end{equation}

If $(\pi,V)$ is a representation of $G$, the Jacquet functor with respect to $N$, $(\pi_M, r_N(V))$, is defined as follows.  As usual,
 $$ V(N)=\langle \pi(n)v-v\mid n\in N, v\in V\rangle, $$
so $V_N=V/V(N)$ is the space of coinvariants.   Let $(\pi_M, r_N(V))$ be a representation of $M$ such that $r_N(V)=V_N$, but the 
 $M$-action is given by
\begin{equation}\label{eq:jacaction}
 \pi_M(m)v = \delta_N^{-1/2}(m)\pi(m)v.
\end{equation}
Since $N$ is normal, it is trivial to see that this is well defined.

Since the induction and the Jacquet functor are normalized, the statement of Frobenius reciprocity is quite simple:
\begin{equation}
 \Hom_G(V,i_P^G(W)) = \Hom_M(r_N(V),W).
\end{equation}

\begin{lemma}\label{lem:main}
Suppose that $G$ is a split reductive group, $K\subset G$ is a maximal compact subgroup and $MN=P\subset G$ is a parabolic subgroup.  Let $K_M=K\cap M$.  For $(\pi,V)$ any smooth representation of $M$, the following statements hold.
\begin{itemize}
 \item[(i)] The map $\varphi:(i_P^G(V))^K\to V^{K_M}$ given by $f\mapsto f(1)$ is an isomorphism.
 \item[(ii)] If $T\in \He$ and $v\in (i_P^G(V))^K$ then $\varphi(T*v)=\cS_M(T)*\varphi(v)$.
\end{itemize}
\end{lemma}
\begin{proof}
Since $G=PK$ any element $g$ in $G$ can be written as $g=mnk$, where $k\in K$, $mn\in P$, and we can define 
 $$ \psi:V^{K_M}\to (i_P^G(V))^K $$
by specifying that $\psi(v)(mnk)=\delta_N(m)^{1/2}\pi(m)v$.  This is well defined precisely because $v\in V^{K_M}$.
Now, (i) follows by computing that $\psi$ is the inverse of $\varphi$. 

For remainder of the proof let $\cS=\cS_M$.  Let $T_\lambda$ be the characteristic function of $K\lambda(\varpi)K$.  Note that $\cS(T_\lambda)$ is determined by its values on a set of coset representatives for $M/K_M$.  We fix such a set.  Using the Iwasawa decomposition $G=PK$, we may write (in analogy to \eqref{eq:double2single})
\begin{equation}\label{eq:mndecomp}
 K\lambda(\varpi)K = \bigcup_{i}\bigcup_{j}m_in_i K 
\end{equation}
with the $m_i$ chosen from the given set of coset representatives.  If $m\notin K\lambda(\pi)K$, then obviously $\cS(T_\lambda)(m)=0$.  Otherwise, $m=m_{j_0}$ for some $m_{j_0}$ appearing in the decomposition \eqref{eq:mndecomp}.  Let
 $$ n(i,j) := \#\{n_i \mid m_jn_iK \mbox{ appears in \eqref{eq:mndecomp}}\}. $$
Then 
\begin{align*}
 \cS(T_\lambda)(m_{j_0}) & = \delta^{1/2}(m_{j_0})\int_{N} T_\lambda(m_{j_0}n)dn \\
 & = \delta^{1/2}(m_{j_0})\sum_{i,j} T_\lambda(m_{j_0}n_i) \\
 & = \delta^{1/2}(m_{j_0}) n(i,j_0),
\end{align*}
since $\vol(N/(K\cap N))=1$.

Let $f\in (\Ind_P^G{V})^K$.  By (i), $\varphi(f)=f(1)=v\in V^{K_M}$.  So, by the previous calculation,
\begin{align*}
 \cS(T_\lambda)*v & = \int_{M} \cS(T_\lambda)(m) \pi(m)v dm \\
  & = \sum_{j} \cS(T_\lambda)(m_j)\pi(m_j)v \\
  & = \sum_j \delta^{1/2}(m_j)n(i,j)\pi(m_j)v.
\end{align*}

On the other hand, since $f$ is fixed by $K$, 
\begin{align*}
 \varphi(T_\lambda*f) = (T_\lambda*f)(1) & = \int_G T_\lambda(g)\sigma(g)f(1)dg \\
  & = \sum_{i,j} f(1m_jn_i) \\
  & = \sum_{j} n(i,j)\delta^{1/2}(m_j)\pi(m_j)v.
\end{align*}
By Proposition~\ref{prop:doublecosets}, $\{T_\lambda\mid \lambda\in \Lambda_c^+\}$ forms a basis of $\He$.  Therefore, (ii) is proved.
\end{proof}

\begin{lemma}\label{lem:borel}
Let $G$ be a reductive group with $P=MN$ the Levi decomposition of a parabolic.  If $V$ is any $G$-module, the map $V^K \to V_N^{K_M}$ is injective.
\end{lemma}
\begin{proof}
For $I\subset G$ the Iwahori subgroup, Borel (see \cite{borel}) proved that $V^I \hookrightarrow V_N^{I_M}$.  As the Jacquet functor is intertwining for the action of $P$, the image of $V^K$ is clearly fixed by $K_M$.  Since $V^K\subset V^I$, this gives the desired result.
\end{proof}

Let $G$  be a split reductive group defined over $F$. If  $\chi: G \rightarrow \GL_1$ is a character then 
$\chi^* : \mathbb C^{\times}  \rightarrow \wh G$ will denote the corresponding co-character.  
 The group  $\chi^*(\mathbb C^{\times})$ is in the center of 
$\wh G$.  For example, if $G=\GL_n$ and $\chi=\det$, then $\chi^*(z)$ is the scalar matrix in $\wh G=\GL_n(\mathbb C)$.

\begin{lemma}\label{lem:firstmain}  Let $\chi: G \rightarrow \GL_1$ be a character.   Let $V$ be
a finite dimensional irreducible representation of $\wh G$ (i.e.\ a Hecke operator for $G$). Let $m$ be a half integer. 
Then $\chi^*(q^m)$ acts on $V$ as $q^n$ for some half integer $n$. Let $\pi$ be a representation of $G$. 
Then $V$ acts on $\pi \otimes |\chi|^m$ as $q^n\cdot V$ acts on $\pi$. 

\end{lemma} 

\begin{lemma}\label{lem:secondmain}

Let $\pi$ be a representation of $G=G'\times G''$ obtained as the pullback of $\pi'$, a representation of $G'$. 
 Let $s\in \wh G''$ be the image of $\smatr{q^{1/2}}{0}{0}{q^{-1/2}}$ under the
 principal $\SL_2\to \wh G''$.  For $V$ a finite dimensional representation of $\wh G$ (i.e.\ a Hecke operator for $G$), write $V=\sum V' \otimes V''$, the restriction of $V$ to $\wh G' \times\wh G''$.  Then $V$ acts on $\pi$ as $\sum \Tr_{V''}(s)V'$ acts on $\pi'$. 
\end{lemma}

\begin{proof}
 Since the Satake parameter of the trivial representation of a group $G''$ is $s$, 
  the result is clear.  
 \end{proof} 

\begin{lemma}\label{lem:thirdmain} 
 Let $\chi: G \rightarrow \mathbb \GL_1$ be a character, and let $C=\chi^*(\C^{\times}) \subseteq \wh G$. 
 Let $\pi'$ be a representation of $\GL_1$. Then $\pi=\pi'\circ \chi$ is a representation of $G$. 
Let $s\in \wh G$ be the image of $\smatr{q^{1/2}}{0}{0}{q^{-1/2}}$ under the principal $\SL_2\to \wh G$.
For $V$ a finite dimensional representation of $\wh G$ (i.e.\ a Hecke operator for $G$), write $V=\sum V' \otimes V''$, the restriction of $V$ to $C\times\SL_2$.  Then $V$ acts on $\pi$ as $\sum \Tr_{V''}(s)V'$ acts on $\pi'$.
\end{lemma} 

\begin{proof} 
We prove this for  $G=\GL_n$ and $\chi=\det$ is the determinant. In this case $C$ is the center. 
 It suffices to prove the statement for the fundamental representations $V_{\lambda_i}=V_i=\wedge^i \C^n$ where 
 $$ \lambda_i=(\underbrace{1,\ldots,1}_{i\mbox{ \scriptsize times}},0,\ldots,0). $$
Since $V_i$ is miniscule, the Satake isomorphism gives $\cS(T_{\lambda_i})=q^{i(n-i)/2}V_i$.  Therefore, the action of $\pi(V_i)$ is 
\begin{align*}
 q^{i(i-n)/2}\int_{K\lambda_i(\varpi)K} \pi(g) dg
 = & q^{i(i-n)/2}\vol(K\lambda_i(\varpi)K)\pi'(\det(\lambda_i(\varpi))) \\
 = & q^{i(i-n)/2}\vol(K\lambda_i(\varpi)K) \pi'(\varpi)^i.
\end{align*}
The center $C$ acts on $V_i$ by the character $z\mapsto z^i$.  Note that
 $$ s = \( \begin{smallmatrix} q^{(n-1)/2} & & & \\ & q^{(n-3)/2} & & \\ & & \ddots & \\ & & & q^{(1-n)/2} \end{smallmatrix}\), $$
is the image of $\smatr{q^{1/2}}{}{}{q^{-1/2}}$ under the principal $\SL_2 \to \wh G$.  So, to complete the proof we just need to show that 
 $$ q^{i(i-n)/2}\vol(K\lambda_i(\varpi)K) = \Tr_{V_i}(s). $$
This is immediate from the discussion in \cite[Section~3]{gross} and the fact that $V_i$ is miniscule.
\end{proof}

\begin{lemma}\label{lem:easymatch}
Let $G$ be a reductive group.  Let $C_c^\infty(G)$ denote the space of smooth, compactly supported functions on $G$. This is a $G \times G$ module for the left and right action of $G$ called the \emph{regular representation}.  On $C_c^\infty(G)^{K \times K}$ we have a matching of the Hecke algebras for the left and right action.
\end{lemma}
\begin{proof}
This is obvious since $C_c^\infty(G)^{K \times K}$ is nothing else  but the Hecke algebra itself. To be precise, since the left action on $f$ in $C_c^\infty(G)$ is by $\lambda_g (f)(x)=f(g^{-1}x)$, a Hecke operator $R$, acting from the right, is matched with $R^*$ defined by $R^*(x)=R(x^{-1})$.
\end{proof}

\begin{remark}
Notice that this matching of Hecke operators, when considered as a matching of virtual representations, matches $V\in R(G)$ with its dual $\wt{V}$.  In particular, if $V$ is self-dual then it is matched with itself.
\end{remark}

\section{Our groups}\label{sec:setup}

\subsection{Octonions}

Let $\Oct$ denote the non-associative division algebra of rank $8$ over $F$.  There is an $F$-linear anti-involution $x\mapsto \bar{x}$ on $\Oct$, hence norm and trace maps
 $$ \N:\Oct\to F\quad x\mapsto x\bar{x} = \bar{x}x,\qquad \Tr:\Oct\to F,\quad x\mapsto x+\bar{x} $$
satisfying
 $$ \N(x\cdot y) = \N(x)\N(y),\qquad \Tr(x\cdot y) = \Tr(y\cdot x), \quad \Tr(x\cdot(y\cdot z)) = \Tr((x\cdot y)\cdot z). $$
On the set $\Oct^0$ of trace zero elements, we have $\bar{x}=-x$.  The group $G_2$ is the automorphism group of $\Oct$.

The quadratic form $\N:\Oct\to F$ has signature $(4,4)$, which means that $\Oct$ has a basis $\{1,i,j,k,l,li,lj,lk\}$.  (Note that $l^2=1$.)  The following basis is particularly useful.
\begin{equation}\label{eq:Octbasis}
\begin{array}{c} s_1 = \frac{1}{2}(i+li),\quad s_2 = \frac{1}{2}(j+lj),\quad s_3 = \frac{1}{2}(k+lk),\quad s_4 = \frac{1}{2}(1+l), \\
 t_1 = \frac{1}{2}(i-li),\quad t_2 = \frac{1}{2}(j-lj),\quad t_3 = \frac{1}{2}(k-lk),\quad t_4 = \frac{1}{2}(1-l). \end{array}
\end{equation}
The multiplication table for this basis is given in Table~\ref{table}.
\begin{table}[h] 
$$\begin{array}{||c||c|c|c||c|c|c||c|c|} \hline
 & s_1 & s_2 & s_3 & t_1 & t_2 & t_3 & s_4 & t_4 \\ \hline \hline
s_1 & 0 & -t_3 & t_2 & s_4 & 0 & 0 & 0 & s_1 \\ \hline
s_2 & t_3 & 0 & -t_1 & 0 & s_4 & 0 & 0 & s_2 \\ \hline
s_3 & -t_2 & t_1 & 0 & 0 & 0 & s_4 & 0 & s_3 \\ \hline \hline
t_1 & t_4 & 0 & 0 & 0 & s_3 & -s_2 & t_1 & 0 \\ \hline
t_2 & 0 & t_4 & 0 & -s_3 & 0 & s_1 & t_2 & 0 \\ \hline
t_3 & 0 & 0 & t_4 & s_2 & -s_1 & 0 & t_3 & 0 \\ \hline \hline
s_4 & s_1 & s_2 & s_3 & 0 & 0 & 0 & s_4 & 0 \\ \hline
t_4 & 0 & 0 & 0 & t_1 & t_2 & t_3 & 0 & t_4 \\ \hline
\end{array}$$
\caption{Multiplication Table for Octonions}
\label{table}
\end{table}

\begin{remark}
From this basis it is evident that a subspace $V\subset \Oct^0$ on which multiplication is trivial is at most 2-dimensional.  (We call such a subspace a \emph{null space} or a \emph{null subspace}.)  Indeed, let $\{i,j,k\}=\{1,2,3\}$.  Then from the multiplication table we see that $s_i^\perp = \langle s_i,t_j,t_k\rangle$, and the null spaces of $\Oct^0$ which contain $s_i$ are all of the form $\langle s_i,at_j+bt_k\rangle$ for fixed $a,b\in F$.  Since $G_2$ acts transitively on (nonzero) elements of trace zero and norm zero, this phenomenon is generic.
\end{remark}

\subsubsection{Maximal parabolic subgroups in $G_2$.}  \label{g2_parabolics} 
They are  described as the stabilizers of null subspaces $V\subset \Oct^0$.  
Let $V_1$ be spanned by $s_1$ and $V_2$ by $s_1$ and $t_2$. Then $V_3=V_1^{\perp}$ is spanned by $s_1,t_2$ and $t_3$. Let 
$P_1=M_1 N_1$ and $P_2=M_2 N_2$ be the stabilizers of $V_1$ and $V_2$, respectively. The Levi factor $M_2$ acts on $V_2$. 
The choice of the basis in $V_2$ gives an isomorphism  $M_2\cong \GL_2$.  The Levi factor $M_2$ acts on $V_3/V_1$ and we have 
an isomorphism $M_1\cong \GL_2$. It is not difficult to see that $g\in \GL_2 \cong M_1$ acts on $V_1$ by $\det(g)$.  The set of all 
$g$ in $G_2$ such that all $s_i$ and $t_i$ are eigenvectors is a maximal split torus $T$ in $G_2$. The modular characters are 
\begin{equation}\label{eq:g2_rho} 
\delta_{N_1}(g)=|\det(g)|^5 \text{ and } \delta_{N_2}(g)=|\det(g)|^3. 
\end{equation}

The stabilizer of $s_4-t_4$ in $G_2$ is a group isomorphic to $\SL_3$. Under the action of this group we have a decomposition 
 \[
 \Oct^0=\langle s_1,s_2,s_3\rangle \oplus \langle t_1,t_2,t_3\rangle\oplus \langle s_4-t_4  \rangle.
 \]   
 We can identify the stabilizer of $s_4-t_4$ with $\SL_3$ so that the action on $\langle s_1,s_2,s_3\rangle$ is standard. The torus 
 $T$ of $G_2$ sits in $\SL_3$. In this way, we can represent elements in $T$ by $3\times 3$ matrices. For example, if 
 $\alpha_l$ is a long root and $\alpha_s$ a short root perpendicular to $\alpha_l$ then, up to Weyl group conjugation,  
 \begin{equation}\label{eq:coroots} 
 \alpha_l^*(t) = \left(\begin{array}{ccc} 
  t &  & \\
   & 1 & \\
   & & t^{-1} \end{array}\right) 
\text{ and } 
\alpha_s^*(t) = \left(\begin{array}{ccc} 
  t &  & \\
   & t^{-2} & \\
   & & t \end{array}\right). 
 \end{equation}

\subsection{Description of groups}

 Let $P=MN$ be a maximal parabolic subgroup of $\bfG$ as in the table below. 
 The group $N$ is abelian, except in the case $E_8$, where $N$ has 
one-dimensional center $Z$.  In order to give a uniform notation, let $Z$ be trivial if $N$ is abelian.  Let $d$ denote the dimension of $N/Z$.
Let $C\cong \mathbb \GL_1$ be the center of $M$.  Fix an isomorphism 
 $\lambda_{*}:  \GL_1 \rightarrow C$ 
such that the adjoint action of $\lambda_{\ast}(z)$ on $N/Z$ is given by multiplication by $z$. 
We have a dual pair $G_2 \times H\subset \bfG$ 
such that $Q=LU = H\cap P$ is a maximal parabolic of $H$. 
Let $N_0\subseteq N/Z$ be the complement of $\bar U/\bar Z$ under the invariant pairing induced by the Killing form. 
We fix an isomorphism of $L$ with a classical group so that the action of $G_2 \times L$ on $N_0$ is isomorphic to 
$\mathbb O_0 \otimes F^n$, the space of $n$-tuples in $\mathbb O_0$, and $h\in L$ acts  on 
an $n$-tuple $(x_1, \ldots , x_n)$ by $(x_1, \ldots , x_n)h^{-1}$.  Since the scalar matrix $z^{-1}$ in $L$ acts on 
$N/Z$ as $z$,  the center of $L$ coincides with the center of $M$. In the case of $\bfG=E_8$, let $i$ be the  
isogeny character of $\GSp_6$. 
 
\begin{table}[h]
$$
\begin{array}{|c||c|c|c|c|} \hline
\bf G & D_5 & E_6 & E_7 & E_8  \\ \hline 
M & D_4 & D_5 & E_6 & E_7 \\ \hline
d & 8 & 16 & 27 & 56 \\ \hline
L & \GL_1 & \GL_{2} & \GL_{3} & \GSp_{6} \\ \hline
\delta_{\bar U} & |\det| & |\det| & |\det|^2 & |i|^8 \\ \hline 
\delta_{\bar N} & |\det|^8 & |\det|^8 & |\det|^9 & |i|^{29} \\ \hline 
\end{array}
$$
\caption{Maximal parabolic subgroups}
\label{table:groupdata}
\end{table}

The last row of the table is the restriction of the character $\delta_{\bar N}$ to $L$. 
The group $\GSp_6$ acts by the isogeny character $i$ on $\bar Z$.   

\subsubsection{Maximal parabolic subgroups in $L$.} \label{l_parabolics} 

 Assume first that $L=\GL_n$. 
Recall that $g\in \GL_n$ acts on $F^n$, the space of $n$-tuples $(x_1, \ldots , x_n)$ by $(x_1, \ldots , x_n)g^{-1}$. 
For $m\leq n$ let  $Q_m$ be the stabilizer of the subspace consisting of the $n$-tuples whose last $n-m$ entries 
are 0.  We have a Levi decomposition $Q_m=L_m U_m$ where  $L_m=\GL_m \times\GL_{n-m}$ is the stabilizer, in $Q_m$,
of  the subspace consisting of the $n$-tuples whose first $m$ entries are 0.
   Let $g=(g_1,g_2) \in \GL_{n-m} \times\GL_m$. The modular character 
$\delta_{U_m}$ is 
\begin{equation}\label{eq:gl_rho} 
\delta_{U_m}(g)=|\det(g_1)|^{m-n}\cdot |\det(g_2)|^{m}. 
\end{equation} 

 Assume now that $L=\GSp_{2n}$.   This is a group of isogenies of a symplectic 
form $(\cdot,\cdot)$ on a $2n$ dimensional space. 
  Let $e_1, \ldots,  e_n, f_1, \ldots , f_n$ be a symplectic basis, i.e. 
\[ 
(e_i,f_j)=-(f_j,e_i)=\delta_{ij} \text{ and } (e_i,e_j)=(f_i,f_j)=0. 
\] 
Using this basis we identify the symplectic space with $F^{2n}$, the space of $2n$-tuples. 
We identify $\GSp_{2n}$ with the group of $2n\times 2n$ matrices  $g$ acting on $F^{2n}$ from the right, 
and preserving the symplectic form, up to an isogeny character : 
\[ 
(vg^{-1},ug^{-1})= i(g)^{-1}(v,u).
\] 

For every $m\leq n$, let $Q_m$ be the subgroup of $\GSp_{2n}$ preserving the subspace spanned by 
$e_1, \ldots, e_m$.  We have a Levi decomposition $Q_m= L_m U_m$ where $L_m$ is the stabilizer, in $Q_m$ of the 
subspace spanned by $f_1, \ldots, f_m$.  
Then $Q_m\cong \GL_m \times \GSp_{2(n-m)}$ where $g=(g_1,g_2) \in \GL_m \times \GSp_{2(n-m)}$
acts as follows: $e_i\mapsto e_i g_1^{-1}$, for $1\leq i \leq m$, $f_i\mapsto i(g_2)^{-1}f_i g_1^{\top}$ for $1\leq i \leq m$  (here $g_1^{\top}$ is the 
transpose of $g_1$) and as 
$g_2^{-1}$ on the remaining $2(n-m)$ basis elements. 
The modular character $\delta_{U_m}$ is 
\begin{equation} \label{eq:sp_rho}
\delta_{U_m}(g)=|\det(g_1)|^{-(2n-m+1)} |i(g_2)|^{\frac{m(2n-m+1)}{2}}. 
\end{equation}

\subsubsection{Example: $\bfG=E_7$}\label{example} 

Let $J_{27}$ be the $27$-dimensional Jordan algebra over $F$ given by 
\begin{equation}\label{eq:Jalg}
 J_{27} = \left\{ \left.A=\(\begin{array}{ccc} a & z & \bar{y} \\ \bar{z} & b & x \\ y & \bar{x} & c \end{array}\) \right\rvert a,b,c\in F,\quad x,y,z\in \Oct\right\}. 
\end{equation}
The determinant on $J_{27}$ gives an $F$-valued cubic form on $J_{27}$. 
The adjoint group $\bfG$ has a maximal parabolic $P=MN$ such that 
\begin{equation}\label{eq:E6nonadjoint}
 M\cong\{g\in \GL(J_{27})\mid \det(g(A))=\lambda(g)\det(A)\mbox{ for some similitude }\lambda(g)\in F^\times\}, 
\end{equation}
 a reductive group of type $E_6$, and  $N\cong J_{27}$ as $M$-modules.  Moreover, $\bar N \cong  J_{27}$, 
 and the natural pairing of $\bar N$ and $N$ can be identified with the trace on $J_{27}$.

 Evidently, $G_2\subset M$ acts term by term on the elements of $A$, and since $g\in \GL_3$ acts via
\begin{equation}\label{eq:gl3action}
 g \cdot A =( \det{g})^{-1} gAg^{t},
\end{equation}
the action of $G_2$ and $\GL_3$ obviously commute, hence $G_2\times \GL_3\subset M$. As described in \cite[Section~5]{ms}, 
$G_2\times \PGSp_6\subset \bfG$ is a dual pair, and $Q=LU=\PGSp_6\cap P$ where $U$ can be identified with $J_6$, 
the subalgebra of $J_{27}$ consisting of (symmetric) matrices with entries in the field $F$, and $L\simeq \GL_3$ with action on $J_{27}\simeq N$ given by \eqref{eq:gl3action}.  Thus, the orthogonal complement $N_0$  of $\bar U$ in $N$ is identified with the subspace 
of $J_{27}$ consisting of matrices with 0 on the diagonal and traceless octonions off the diagonal, that is, 
 $N_0$ is identified with the set of triples $(x,y,z)$ of traceless octonions. Moreover, 
 $(g, h) \in G_2 \times \GL_3$ acts on $(x,y,z)$  by $(gx,gy,gz) h^{-1}$.  (The action of $h$ follows from Cramer's rule.)

\section{The proof}\label{sec:general}

\subsection{The base case}\label{sec:basecase}

Let $\bfV$ be the minimal representation of $D_4$. Let $G'= S_3$ be the group of permutations of 3 letters. Then $S_3$ acts on $D_4$, by outer automorphisms, fixing $G_2$. Since $\bfV$ can be extended to a representation of a semi-direct product of $D_4$ and $S_3$, we have a dual pair $G\times G'= G_2\times S_3$ acting on $\bfV$.  We let $\wh{G}'=S_3$ and 
 $$ r :S_3 \rightarrow G_2(\mathbb C) $$
such that the centralizer $C$ of $r(S_3)$ in $G_2(\C)$ is $\SO(3) \subset \SL_3(\C) \subset G_2(\C)$.  (The group $\SO(3) \simeq \PGL_2(\C)$ corresponds to the subregular unipotent orbit by the Jacobson-Morozov theorem.)

Let $K'=G'$. Then the Hecke algebra $\He'$ is one-dimensional.  With these choices, Theorem~\ref{theorem} asserts that the $S_3$-invariants of the minimal representation of $D_4$ is the unramified representation $\pi_{sr}$ of $G_2$ whose Satake parameter corresponds to the subregular orbit. This is proved in \cite{hms}.

\subsection{The general case}\label{sec:generalcase}

Assume that $\bfG\neq D_4$. Then we have a maximal parabolic $P=MN$ in $\bfG$ and the corresponding maximal 
parabolic $Q=LU$ in $H$ as in Table~\ref{table:groupdata}.
For simplicity, let $K$ be the maximal compact subgroup of $G_2$, and $K'$ that of $H$. Assume that we want to show 
matching of two operators $T$ and $T'$. 
 By Lemma~\ref{lem:borel},
 $$ \bfV^{K\times K'} \hookrightarrow r_{\bar U}(\bfV)^{K\times K'_L} $$
where $K'_L=K'\cap L$.  Thus Theorem~\ref{theorem} holds if we can show matching 
of $T$ and $\mathcal S_L(T')$ on $r_{\bar U}(\bfV)^{K\times K'_L}$.  If $\bfG\neq E_8$, the unnormalized Jacquet functor $\bfV_{\bar U}$ was 
computed in \cite{ms}.   In the context of the present work, we find it convenient to describe these results in terms of maximal parabolic 
subgroups $Q_m=L_mU_m$ of $L\cong \GL_n$  and maximal parabolic subgroups $P_m=M_mN_m$ of $G_2$ 
(as defined in Sections~\ref{l_parabolics} and~\ref{g2_parabolics}, respectively). In particular, we have fixed 
isomorphisms  $L_m\cong \GL_m \times \GL_{n-m}$ and $M_m\cong \GL_2$. 

Let $s,t$ be a pair of real numbers. For $m=1,2$, let
$C_c^{\infty}(\GL_m)[s,t]$ be the vector space $C_c^{\infty}(\GL_m)$ with an $M_m\times L_m$-module structure defined by
 $$ (g_1,g_2)\cdot f(h) = \abs{\det{g_1}}^s\abs{\det{g_2}}^t f(hg_1) $$
for any $(g_1,g_2)\in L_m$, and by  
 \begin{equation}\label{eq:levi_action}
g\cdot f(h) = \left\{\begin{array}{cc} f(\det{g}^{-1} h) & \mbox{ if }m=1, \\  f(g^{-1} h) & \mbox{ if }m=2, \end{array} \right.
\end{equation} 
 for any $g\in M_m$. 

We shall omit $[s,t]$ in the notation if $s=t=0$.
With this notation in hand, we now describe the Jacquet module $r_{\bar{U}}(\bfV)$ for each of our cases.

\subsubsection{$\bfG=D_5$}

\begin{proposition}\label{prop:msD5} As a $G_2\times \GL_1$-module, $r_{\bar U}(\bfV)$ has a filtration with two successive 
sub quotients
\begin{enumerate} 
\item   $i^{G_2}_{P_1}(C_c^\infty(\GL_1))$. 
\item $\bfV(M)\otimes |\det|^{\frac{1}{2}} \oplus 1\otimes |\det|^{\frac{5}{2}}$.
\end{enumerate} 
Here $\bfV(M)$ is the minimal representation of $M/C$. 
\end{proposition}
\begin{proof} 
This is simply a normalized version of Proposition~2.3 of \cite{ms}. Indeed,  $\bfV_{\bar U}$ has a filtration with two successive 
quotients
\begin{itemize} 
\item  $\Ind_{P_1}^{G_2}(C_c^\infty(\GL_1)) \otimes |\det|^3$. 
\item $\bfV(M)\otimes |\det| \oplus 1\otimes |\det|^{3}$. 
\end{itemize} 
where induction is not normalized. Since $r_{\bar U}(\bfV)=\bfV_{\bar U}\otimes \delta_{\bar U}^{-\frac12}$,  
and $\delta_{\bar U}=|\det|$ by Table \ref{table:groupdata}, $r_{\bar U}(\bfV)$ has 
a filtration with two successive quotients
\begin{itemize} 
\item  $\Ind_{P_1}^{G_2}(C_c^\infty(\GL_1)) \otimes |\det|^{\frac52}$. 
\item $\bfV(M)\otimes |\det|^{\frac12} \oplus 1\otimes |\det|^{\frac52}$. 
\end{itemize} 
and (2) follows. 
Since, for any $m$ and $s$, $C_c^\infty(\GL_m) \cong C_c^\infty(\GL_m)\otimes |\det|^s$ as $\GL_m \times \GL_m$-modules,  we can replace $C_c^\infty(\GL_1)$ in the first bullet by $C_c^\infty(\GL_1)\otimes |\det|^{-\frac{5}{2}}$.  
By \eqref{eq:usualind} and \eqref{eq:g2_rho}, this normalizes the induction for $G_2$ and, at the same time, removes the character  $|\det|^{\frac52}$  of $\GL_1$.  Hence (1) follows.  
\end{proof} 
\medskip 

\begin{proof}[Proof of Theorem~\ref{theorem} in case $\bfG=D_5$.]
The dual group of  $H=\PGL_2$  is  $\SL_2(\C)$. The map $r:\SL_2(\C)\to G_2(\C)$ corresponds to a long root
$\alpha_l$  of $G_2$: $r(\SL_2(\C))=\SL_{2,l}(\C)\subset G_2(\C)$.  The centralizer $C$  of $\SL_{2,l}(\C)$ in $G_2(\C)$  
is  $\SL_{2,s}(\C)$ corresponding to a short root $\alpha_ s$ perpendicular to $\alpha_l$. 

Let $V$ be a finite dimensional representation of $G_2(\C)$ and $T_2$ the corresponding Hecke operator for $G_2$.
Let  
\begin{equation}\label{eq:s}
s= \alpha^*_s(q^{1/2})=\smatr{q^{1/2}}{0}{0}{q^{-1/2}} \in \SL_{2,s}(\C).
\end{equation} 
  If the restriction of $V$ to $\SL_{2,l}(\C)\times \SL_{2,s}(\C)$ is 
$\sum V'\otimes V''$,  we define $T_1$ as corresponding to $\sum \Tr_{V''}(s) V'$.  We want to show that $T_2$ matches with 
$\mathcal S_L(T_1)$ on $r_{\bar U}(\bfV)$.  Since $\wh L$ is the torus $\alpha^*_l(\C^{\times})\subseteq \SL_{2,l}(\C)$, the operator 
 $\mathcal S_L(T_1)$ corresponds to the representation $\sum \Tr_{V''}(s) V'$ of $\alpha^*_l(\C^{\times})$ 
 obtained by restricting each $V'$  to the torus $\alpha^*_l(\C^{\times})$.  
 
 First, we show  matching on (1) in Proposition \ref{prop:msD5}. By Lemma \ref{lem:main}, we need to show  matching of 
 $\mathcal S_{M_1}(T_2)$ and $\mathcal S_L(T_1)$ on $C_{c}^{\infty}(\GL_1)$.  The operator $\mathcal S_{M_1}(T_2)$
 corresponds to the restriction of $V$ to $\wh{M_1}\simeq \GL_{2,s}(\C)$, the dual group of $M_1$.
   The center of $\GL_{2,s}(\C)$ is the torus 
$\alpha^*_l(\C^{\times})\subseteq \SL_{2,l}(\C)$. Let  $s$  be as in \eqref{eq:s}  and 
 let $V=\sum V'\otimes V''$ be the restriction of $V$ to $\SL_{2,l}(\C)\times \SL_{2,s}(\C)$, as before. 
By Lemma \ref{lem:thirdmain},  $\mathcal S_{M_1}(T_2)$ acts on 
$C_{c}^{\infty}(\GL_1)$ as the Hecke operator for $\GL_1$ that corresponds to the representation $\sum \Tr_{V''}(s) V'$
 of $\alpha^*_l(\C^{\times})$, the center of  $\GL_{2,s}(\C)$.  In particular, this is the same $\GL_1$-operator as $\mathcal S_L(T_1)$. 
Matching now follows from Lemma \ref{lem:easymatch}.

Using the previously proved base case, matching on (2) in Proposition \ref{prop:msD5} reduces to a simple check on 
two $G_2 \times L$ modules: $\pi_{sr}\otimes |\det|^{\frac12}$ and $1\otimes |\det|^{\frac52}$. 
\end{proof}

\subsubsection{$\bfG=E_6$}

\begin{proposition}\label{prop:msE6} As a $G_2\times \GL_2$-module, $r_{\bar U}(\bfV)$ has a filtration with three successive 
sub quotients
\begin{enumerate} 
\item $i^{G_2}_{P_2}(C_c^\infty(\GL_2))$.
\item   $i^{G_2\times \GL_2}_{P_1\times Q_1}(C_c^\infty(\GL_1)[-\frac12,1])$. 
\item $\bfV(M)\otimes |\det|^{\frac{1}{2}} \oplus 1\otimes |\det|^{\frac{3}{2}}$.
\end{enumerate} 
Here $\bfV(M)$ is the minimal representation of $M/C$. 
\end{proposition}
This is a normalized version of Theorem~4.3 of \cite{ms} which states that  $\bfV_{\bar U}$
 has a filtration with three successive quotients
\begin{itemize} 
\item  $\Ind_{P_2}^{G_2}(C_c^\infty(\GL_2)) \otimes |\det|^2$. 
\item  $\Ind_{P_1\times Q_1}^{G_2\times \GL_2}(C_c^\infty(\GL_1)) \otimes |\det|^2$. 
\item $\bfV(M)\otimes |\det| \oplus 1\otimes |\det|^{2}$. 
\end{itemize} 

\begin{proof}[Proof of Theorem~\ref{theorem} in case $\bfG=E_6$.]
The dual group of  $H=\PGL_3$  is  $\SL_3(\C)$. We have an inclusion $r:\SL_3(\C)\to G_2(\C)$ where 
$\SL_{3}(\C)\subset G_2(\C)$ is given by the long roots. 

Let $V$ be a finite dimensional representation of $G_2(\C)$ and $T_2$ the corresponding Hecke operator for 
$G_2$. We restrict $V$ to $\SL_3(\C)$, and let $T_2$ be the corresponding Hecke operator for $\PGL_3$. 
We want to show that $T_2$ matches with $\mathcal S_L(T_1)$ on $r_{\bar U}(\bfV)$.  

First, we consider matching on $\bfV(M)\otimes |\det|^{\frac12}$. The dual group of $L=\GL_2$ is $\GL_{2,l}(\C) \subseteq \SL_3(\C)$. 
The group $L$ acts on $\bfV(M)$ by its quotient $\PGL_2$. The dual group of $\PGL_2$ is $\SL_{2,l}(\C)$.  Thus, 
$\mathcal S_{L}(T_1)$ acts on $\bfV(M)$ as the Hecke operator for $\PGL_2$ that corresponds to the restriction of $V$ to 
$\SL_{2,l}(\C)$. We need to take into account the twist by $|\det|^{\frac12}$. 
Let $\chi$ be the determinant character of $L$.  Let $\chi^*: \C^{\times}\rightarrow \GL_{2,l}(\C)$ be the corresponding co-character. 
Note that $\chi^*=\alpha^*_s$.  Let $s=\alpha^*_s(q^{1/2}) \in \SL_{2,s}(\C)$.  
Let $V=\sum V'\otimes V''$ be the restriction of $V$ to $\SL_{2,l}(\C)\times \SL_{2,s}(\C)$.  Let $T$ be the Hecke operator 
for $\PGL_2$ that corresponds to the representation $\sum \Tr_{V''}(s) V'$ of $\SL_{2,l}(\C)$. 
By Lemma \ref{lem:firstmain}, $\mathcal S_L(T_1)$ acts on 
$\bfV(M)\otimes|\det|^{\frac12}$ as $T$ acts on $\bfV(M)$. But $T$ is matched with $T_2$ on $\bfV(M)$, by the case $\bfG=D_5$. 

To prove matching on  (1) in Proposition \ref{prop:msE6}  it suffices to show that 
$\mathcal S_{M_2}(T_2)$ and $\mathcal S_L(T_1)$ are matching on $C_c^{\infty}(\GL_2)$, by Lemma \ref{lem:main}. 
Since the dual group of $M_2$ is conjugated in $G_2(\C)$ to $\GL_{2,l}(\C)$, the dual group of $L$, matching of 
$\mathcal S_{M_2}(T_2)$ and $\mathcal S_L(T_1)$ follows from Lemma \ref{lem:easymatch}. 
(See also the remark following Lemma \ref{lem:easymatch}.) 

Matching on (2) is similar to (1), albeit slightly more complicated to write down, so we omit details.  
\end{proof}

\subsubsection{$\bfG=E_7$} 

\begin{proposition}\label{prop:msE7} As a $G_2\times \GL_3$-module, $r_{\bar U}(\bfV)$ has a filtration with three successive 
sub quotients
\begin{enumerate} 
\item $i^{G_2\times \GL_3}_{P_2\times Q_2}(C_c^\infty(\GL_2))$.
\item   $i^{G_2\times \GL_3}_{P_1\times Q_1}(C_c^\infty(\GL_1)[-\frac12,\frac12])$. 
\item $\bfV(M) \oplus 1\otimes |\det|$.
\end{enumerate} 
Here $\bfV(M)$ is the minimal representation of $M/C$. 
\end{proposition}
This is a normalized version of Theorem~5.3 of \cite{ms} which states that  $\bfV_{\bar U}$
 has a filtration with three successive quotients
\begin{itemize} 
\item  $\Ind_{P_2\times Q_2}^{G_2\times \GL_3}(C_c^\infty(\GL_2)) \otimes |\det|^2$. 
\item  $\Ind_{P_1\times Q_1}^{G_2\times \GL_3}(C_c^\infty(\GL_1)) \otimes |\det|^2$. 
\item $\bfV(M)\otimes |\det| \oplus 1\otimes |\det|^{2}$. 
\end{itemize} 


\begin{proof}[Proof of Theorem~\ref{theorem} in case $\bfG=E_7$.]
The dual group of $H=\PGSp_6$ is $\Spin_7(\mathbb C)$.  
Let $\mathbb Z^3$ be the root  lattice of $\Spin_7(\mathbb C)$ so that the short roots correspond to the standard 
basis vectors $e_1, e_2, e_3$  in $\mathbb Z^3$. 
Let $V_8$ be the spin representation of $\Spin_7(\C)$. The weights of $V_8$ are $ (\pm \frac12, \pm \frac12, \pm\frac12)$. 
Under the action of $\wh L\cong \GL_3(\C)$ the spin representation decomposes as 
\[ 
V_8 = V_1 \oplus V_3 \oplus V_3^* \oplus V^*_1 
\] 
where $V_3$ is the standard representation of $\GL_3(\C)$ and $V_1$ is the determinant character. The weights of
these 4 summands are $(x,y,z)=(\pm \frac12, \pm \frac12, \pm\frac12)$ such that $x+y+z=\frac32,\frac12,-\frac12,-\frac32$, respectively. 
We have an injection  
\[
r: G_2(\C) \rightarrow \Spin_7(\C)
\]
 where $G_2(\C)$ is defined as the stabilizer of a non-zero 
vector in $V_1$, for example. In particular, $G_2(\C) \cap \GL_3(\C)=\SL_3(\C)$. 

Let $V$ be a finite dimensional representation of $\Spin_7(\C)$ and $T_1$ the corresponding Hecke operator for 
$\PGSp_6$. We restrict $V$ to $G_2(\C)$, and let $T_2$ be the corresponding Hecke operator for $G_2$. 
We want to show that $T_2$ matches with $\mathcal S_L(T_1)$ on $r_{\bar U}(\bfV)$. Matching on (3) in Proposition \ref{prop:msE7} 
trivially follows from the previously proved case $\bfG=E_6$. 

To prove matching on (1) it suffices to show that $\mathcal S_{M_2}(T_2)$ matches 
with $\mathcal S_{L_2}\circ \mathcal S_L(T_1)$ on $C_c^{\infty}(\GL_2)$.  We can assume that the dual group of  $M_2$ is
$\GL_{2,l}(\C)\subseteq \SL_3(\C)$ where 
$g\in\GL_{2,l}(\C)$ sits in $\SL_3(\C)$ as a block diagonal matrix $(g,\det g^{-1})$. 
The operator $\mathcal S_{M_2}(T_2)$ corresponds to the restriction of $V$ to $\GL_{2,l}(\C)$.  
On the other hand,  $L_2=\GL_2 \times \GL_1$ 
and  $\mathcal S_{L_2}\circ \mathcal S_L(T_1)$  acts on 
$C_c^{\infty}(\GL_2)$ as the Hecke operator for $\GL_2$ that corresponds to the restriction of $V$ to the first factor of  
\[ 
\wh L_2=\GL_2(\C)\times \GL_1(\C)\subseteq \GL_3(\C)\subseteq \Spin_7(\C). 
\] 
 The first factor of $\wh L_2$ is conjugated to $\GL_{2,l}(\C)$ in $\Spin_7$ by the reflection corresponding to the short root $e_3$. Matching on (1) now follows from Lemma \ref{lem:easymatch}. 

To prove matching on (2) it suffices to show that $\mathcal S_{M_1}(T_2)$ matches 
with $\mathcal S_{L_1}\circ \mathcal S_L(T_1)$ on $C_c^{\infty}(\GL_1)$. 
We have $\wh{M_1}\simeq \GL_{2,s}(\C)$, where the center of $\GL_{2,s}(\C)$ is the torus 
 $\alpha^*_l(\C^{\times})\subseteq \SL_{2,l}(\C)$. By Lemma \ref{lem:thirdmain}, $\mathcal S_{M_2}(T_2)$ acts on $C_c^{\infty}(\GL_1)$ 
 as the Hecke operator for $\GL_1$ that corresponds to the restriction of $V$ to  $\alpha^*_l(\C^{\times})$, 
 weighted by the eigenvalues of $\alpha^*_s(q^{1/2})$. 
On the other hand, $\mathcal S_{L_1}\circ \mathcal S_L(T_1)$  acts on 
$C_c^{\infty}(\GL_1)$ as the Hecke operator for $\GL_1$ that corresponds to the restriction of $V$ to 
$\GL_1(\C)$, the first factor of 
\[ 
\wh L_1=\GL_1(\C)\times \GL_2(\C)\subseteq \GL_3(\C)\subseteq \Spin_7(\C),
\] 
  weighted by the eigenvalues of 
  \[ 
   s=
   \left(\begin{array}{ccc} 
  q^{-\frac12} &  & \\
   & q^{\frac12} & \\
   & & q^{\frac12} \end{array}\right)\cdot 
   \left(\begin{array}{ccc} 
  1 &  & \\
   & q^{\frac12} & \\
   & & q^{-\frac12} \end{array}\right)\in\GL_3(\C) \subseteq \Spin_7(\C), 
   \] 
where the first matrix in the above product reflects the twisting $[-\frac12,\frac12]$ in (2), and the second comes from 
Lemma \ref{lem:secondmain}.
The pairs $(\alpha^*_l (\C^{\times}), \alpha_s^*(q^{1/2}))$ (see \eqref{eq:coroots}) and $(\GL_1(\C), s)$ are conjugated in $\Spin_7(\C)$ 
 by the reflection corresponding to the short root $e_3$. Matching on (2) now follows from Lemma \ref{lem:easymatch}.
\end{proof}

\subsubsection{$\bfG=E_8$}\label{sp_parabolic}

Let $Q_m=L_mU_m$ be the maximal parabolic subgroup of $L\cong\GSp_6$ (as in Section~\ref{l_parabolics}). 
Let $s,t$ be a pair of real numbers. For $m=1,2$, let
$C_c^{\infty}(\GL_m)[s,t]$ be the vector space $C_c^{\infty}(\GL_m)$, with an $M_m\times L_m$-module 
structure defined by  
 $$ (g_1,g_2)\cdot f(h) = \abs{\det{g_1}}^s \abs{i(g_2)}^t f(hg_1)$$
 for any $(g_1,g_2)\in L_m\cong \GL_m \times\GSp_{6-2m}$ and by \eqref{eq:levi_action} for any $g\in M_m$. 

\begin{proposition}\label{prop:swE8} As a $G_2\times \GSp_6$-module, $r_{\bar U}(\bfV)$ has a filtration with three successive 
sub quotients
\begin{enumerate} 
\item $i^{G_2\times \GSp_6}_{P_2\times Q_2}(C_c^\infty(\GL_2)[1,-\frac32])$.
\item   $i^{G_2\times \GSp_6}_{P_1\times Q_1}(C_c^\infty(\GL_1)[ \frac12,-\frac12])$. 
\item $\bfV(M)\otimes |i|^{-1}  \oplus 1\otimes |i|$.
\end{enumerate} 
Here $\bfV(M)$ is the minimal representation of $M/C$. 
\end{proposition}

The proof of this proposition is given in Section~\ref{sec:E8Jac}.

\begin{proof}[Proof of Theorem~\ref{theorem} in case $\bfG=E_8$.]
The map $r: G_2(\C) \rightarrow F_4(\C)$ is described as follows.  A split,  simply connected group $G_{sc}$ of type $E_6$ can be 
realized as a subgroup of  $\GL(J_{27})$  fixing $\det : J_{27} \rightarrow F$. As described in \ref{example},  there is a dual pair 
$G_2 \times SL_3\subseteq G_{sc}$. The group $F_4$ is the subgroup of $G_{sc}$ consisting  of  elements fixing the 
identity matrix in $J_{27}$. It is easy to check that $(G_2\times \SL_3)\cap F_4=G_2 \times \SO_3$. This defines $r$, and 
 the centralizer $C$ of $r(G_2(\C))$ is $\SO_3(\C)$.  The dual group $\wh L$ of $L=\GSp_6$ is a Levi factor of type $B_3$. 
 Let $i^* : \C^{\times}\rightarrow  \wh L$ be the co-character corresponding to the isogeny character $i$ of $\GSp_6$. Then 
 $i^*(\C^{\times})$ is the center of $\wh L$.  We can conjugate $G_2(\C)$ in $F_4(\C)$ so that 
 \[ 
 G_2(\C) \subseteq \Spin_7(\C) = [\wh L,\wh L]. 
 \] 
 In this way, $i^*(\C^{\times})$ is a maximal torus in $\SO_3(\C)$, the centralizer of $G_2(\C)$. The image 
 of $\smatr{q^{1/2}}{0}{0}{q^{-1/2}} \in \SL_2(\C)$ in $\SO_3(\C)$ is $s=i^*(q)$. 

Let $V$ be a finite dimensional representation of $F_4(\C)$ and $T_1$ the corresponding Hecke operator for 
$F_4$. If $V=\sum V'\otimes V''$ is the restriction of $V$  to $G_2(\C)\times \SO_3(\C)$, we have defined $T_2$ as corresponding to 
$\sum \Tr_{V''}(s) V'$. 
We want to show that $T_2$ matches with $\mathcal S_L(T_1)$ on $r_{\bar U}(\bfV)$.   The operator $\mathcal S_L(T_1)$
corresponds to the representation $\wh L$ obtained by restricting $V$ to $\wh L$. We now show matching on (3).  
Let $V=\sum V_n$ be the restriction of $V$ to $\Spin_7(\C)=[\wh L,\wh L]$, where $i^*(q)$ act as $q^n$ on $V_n$. 
Let $T$ be the Hecke operator for $\PGSp_6$ that corresponds to the representation $\sum q^{-n} V_n$ of $\Spin_7(\C)$. 
Then, 
by Lemma \ref{lem:firstmain}, $\mathcal S_L(T_1)$ acts on $\bfV(M)\otimes |i|^{-1}$  as $T$ acts on $\bfV(M)$. Matching
on  $\bfV(M)\otimes |i|^{-1}$  now follows from the previously proved  case $\bfG=E_7$.  Let $s_p\in G_2(\C)$ be the image of 
$\smatr{q^{1/2}}{0}{0}{q^{-1/2}} \in \SL_2(\C)$ under the principal 
\[ 
f:\SL_2(\C)\rightarrow G_2(\C).
\] 
 Then $T_2$ acts on $1\otimes |i|$ as 
the scalar $\sum \Tr_{V''}(s)\cdot \Tr_{V'}(s_p)$. Under the inclusion $G_2(\C) \subseteq \Spin_7(\C)$,  the composite 
$f: \SL_2(\C)\rightarrow \Spin_7(\C)$ is the principal $\SL_2(\C)$ in $\Spin_7(\C)$.
 By Lemma  \ref{lem:firstmain}, $\mathcal S_L(T_1)$ acts on 
$1\otimes |i|$ as the scalar $\sum q^n\Tr_{V_n}(s_p)$. Since 
\[ 
\sum q^n\Tr_{V_n}(s_p)= \sum \Tr_{V''}(s)\cdot \Tr_{V'}(s_p), 
\] 
 matching is now proved on (3). The remaining 
cases are similar to $\bfG=E_7$, so we leave them as an exercise. 
\end{proof}

\section{A Jacquet module for $E_8$}\label{sec:E8Jac}

In this section we prove Proposition \ref{prop:swE8}.  In this case $N$ is a Heisenberg group. 
A starting point is the following (Theorem 6.1 \cite{ms}). 

\begin{proposition}  \label{prop:msE8} Let $\Omega$ be the $M$-orbit  of the highest weight vector in $N/Z$.  We have the 
following exact sequence of  $\bar P$-modules, 
\[ 
0\rightarrow C_c^{\infty}(\Omega)\rightarrow \bfV_{\bar Z}\rightarrow  \bfV_{\bar N}\rightarrow 0.
\]
The action of $\bar P$ on  $f\in C_c^{\infty}(\Omega)$ is given by: 
\begin{itemize} 
\item For every  $\bar n\in \bar N$
\[ 
\Pi(n)f(x)=\psi(\langle  x, \bar n\rangle )f(x).
\]  
\item For every  $m\in  M$ 
\[
\Pi(m)f(x)= |i(m)|^5 f(m^{-1} x m).
\]  
\end{itemize} 
Here $\psi$ is a non-trivial additive character, $\langle\cdot,\cdot\rangle$ is a pairing between $N/Z$ and $\bar N/\bar Z$ induced by the 
Killing form, and $i: M \rightarrow \GL_1$ 
is the character obtained by acting on $\bar Z$. 
Moreover, $\bfV_{\bar N} \cong \bfV(M)\otimes |i|^3 \oplus  |i|^5 $ where $\bfV(M)$ is the minimal representation of $M$ with the center acting trivially. 

\end{proposition} 

By Section 7 in \cite{ms} we have an identification of vector spaces 
\[ 
N/Z\cong \bar N/\bar Z\cong F \oplus J_{27} \oplus J_{27} \oplus F
\] 
so that the pairing $\langle\cdot,\cdot\rangle$ is given by 
\[ 
\langle (a, B, C,d), (a',B',C',d') \rangle = aa'+ \Tr(BB')+ \Tr(CC') + bb'. 
\] 
Under these identifications, the action of $G_2$ on $N/Z$ is the obvious one, and the centralizer of $G_2$ 
in $\bar N/\bar Z$ is 
\[ 
\bar U/\bar Z \cong F \oplus J_{6} \oplus J_{6} \oplus F
\] 
where $J_6$ is the space of $3\times 3$ matrices with coefficients in $F$.  It follows that the orthogonal complement of $\bar U/\bar Z$ 
in $N/Z$ can be identified by $J^0 \oplus J^0$ where  $J^0$ is the set of $B\in J$ of the form
 $$ B = \(\begin{array}{ccc} 0 & z & -y \\ -z & 0 & x \\ y & -x & 0 \end{array}\) $$
with $x,y,z\in \Oct^0$.  Given this, we may denote an element  $ (B,B')\in J^0\oplus J^0$
  by a six-tuple  
  \[
  (u,u')=((x,y,z),(x',y',z'))
  \]
  of traceless octonions.  
The action of $G_2\times \GSp_6$ on these elements is simple to describe.  First, $g\in G_2$ acts on every component 
of the six-tuple $(u,u')$ from the left, and $h\in \GSp_6$ acts by $(u,u') h^{-1}$. In this way, the character $i$ of $M$ restricts to the 
isogeny character of $\GSp_6$. We highlight the action of certain subgroups: 
$\SL_2\times \SL_2\times \SL_2\subseteq \Sp_6$ acting on the pairs $(x,x')$, $(y,y')$ and $(z,z')$ respectively in the obvious way, 
and $h\in \GL_3$ by
 $$ (uh^{-1},u'h^{t}). $$
Note that $\GSp_6$ preserves, up to the isogeny character, 
 $$ J^0\oplus  J^0\to \wedge^2\Oct^0\qquad ((x,y,z),(x',y',z'))\mapsto x\wedge x' + y\wedge y' + z\wedge z'. $$
 
 Let $\Omega_0$ be the intersection of $\Omega$ with $J^0\oplus J^0$, the orthogonal complement of $\bar U/\bar Z$.  It follows, 
from Proposition  \ref{prop:msE8}, that there is an exact sequence of $G_2 \times \GSp_6$-modules 
 \[ 
0\rightarrow C_c^{\infty}(\Omega_0)\rightarrow \bfV_{\bar U}\rightarrow  \bfV_{\bar N}\rightarrow 0, 
\]
where $(g, h)\in G_2 \times \GSp_6$ acts on $f \in C_c^{\infty}(\Omega_0)$ by 
\[ 
\Pi(g,h)f(x)= |i(h)|^5 f(g^{-1} x h). 
\] 
In order to understand $C_c^{\infty}(\Omega_0)$, we need to compute $G_2 \times \GSp_6$-orbits on $\Omega_0$.

\begin{proposition}\label{prop:E8Omega0}
The set $\Omega_0$ consists of pairs of $((x,y,z),(x',y',z'))\in J^0\times J^0$ such that $\langle x,y,z,x',y',z'\rangle$ is
 a non-zero null subspace of $\Oct^0$, and such that
 $$ x\wedge x' + y\wedge y' + z\wedge z'=0. $$
Moreover, $\Omega_0$ consists of two $G_2\times \GSp_6$ orbits $\Omega_1$ and $\Omega_2$ where
 $$ \Omega_m = \{((x,y,z),(x',y',z'))\in \Omega\mid \dim(\langle x,y,z,x',y',z'\rangle)=m\}. $$
\end{proposition}

\begin{proof} If $ (0,B,B',0)\in \Omega_0$ then, by Lemma 7.5 in \cite{ms},  $B$ is a rank one matrix, $B^2= \Tr(B) B$. 
Since $\Tr(B)=0$, we have $B^2=0$, and this is equivalent to  
 $$ x^2=y^2=z^2=xy=yz=0, $$
 i.e. the entries of $B$ span a null subspace of $\Oct^0$. 
 Acting by $\SL_2\times \SL_2\times \SL_2$, we can replace $x$, $y$ and $z$ (all or some) by $x'$, $y'$ and $z'$, 
 respectively. Hence $\langle x,y,z,x',y',z'\rangle$ is a non-zero null subspace of $\Oct^0$. 
 If the dimension of this null space is 1 then, without loss of generality, we can assume that $x\neq 0$. 
  Then $(u,u')$ is in the $\GSp_6$ orbit of $((x,0,0),(0,0,0))$. Since $G_2$ acts transitively on 1-dimensional null subspaces, we 
  have one $G_2\times \GSp_6$ orbit. 
 If the dimension of the null space  is 2 then, without loss of generality, we can assume that 
 $x$ and $z$ are a basis of this space. Using the action of $\GL_3$ we can arrange that $y=0$. Since
 \[ 
 x'=ax+cz, y'=ex+ fz, z'=bz + dx 
 \] 
 for some $a,b,c,d,e,f\in F$, 
 $$ x\wedge(ax+cz) + 0\wedge (ex+fz) + z\wedge (bz+cx) = (c-d)(x\wedge z) $$
 and this is 0 if and only if $c=d$. If $c=d$ then it is not too difficult to see that $(u,u')$ is in the $\GSp_6$ orbit of $((x,0,z),(0,0,0))$. 
 Since $G_2$ acts transitively on 2-dimensional null subspaces,  we have one $G_2\times \GSp_6$ orbit.  Thus, to finish the proof 
 we must show that $c=d$. This is done in Lemma \ref{lem:crossprodcondition}, using that $B'= A\times B$ (the cross product) 
 for some $A\in J_{27}$, by Lemma 7.5 in \cite{ms}. 
\end{proof}

\begin{lemma}\label{lem:crossprodcondition}
Suppose that $x,z\in \Oct^0$ be linearly independent such that $x^2=z^2=xz=0$.  Let $x_1,y_1,z_1\in \Oct^0$, and set
 $$ A=A_0+\(\begin{smallmatrix} 0 & z_1 & -y_1\\ -z_1 & 0 & x_1 \\ y_1 & -x_1 & 0 \end{smallmatrix}\),\qquad B=\(\begin{smallmatrix} 0 & z & 0\\ -z & 0 & x \\ 0 & -x & 0 \end{smallmatrix}\)$$
where $A_0\in J_6$.  If
 $$ A\times B =  \(\begin{smallmatrix} 0 & z' & -y'\\ -z' & 0 & x' \\ y' & -x' & 0 \end{smallmatrix}\) $$
is such that $x',y',z'\in \Oct^0$ then $x',y',z'\in Fx+Fz$.  Moreover,
\begin{equation}\label{eq:crossprodcondition}
 x'=bx+cz,\qquad\mbox{and}\qquad z' = az + cx
\end{equation}
for some constants $a,b,c\in F$.
\end{lemma}
\begin{proof}
Since $G_2$ acts transitively on the set of $2$-dimensional null spaces of $\Oct^0$, by the $G_2$ action (which commutes with the cross product), we may assume that $x=s_1$ and $z=t_2$, and let
\begin{gather*}
 x_1 = a_1^x s_1 + a_2^x s_2 + a_3^x s_3 + b_1^x t_1 + b_2^x t_2 + b_3^x t_3 + c^x (s_4-t_4), \\
 y_1 = a_1^y s_1 + a_2^y s_2 + a_3^y s_3 + b_1^y t_1 + b_2^y t_2 + b_3^y t_3 + c^y (s_4-t_4), \\
 z_1 = a_1^z s_1 + a_2^z s_2 + a_3^z s_3 + b_1^z t_1 + b_2^z t_2 + b_3^z t_3 + c^z (s_4-t_4) \\
\end{gather*}
where the elements $s_i,t_j\in \Oct$ are the basis elements given in \eqref{eq:Octbasis}.

The cross product is given by
 $$ A\times B = A\circ B - \textstyle{\frac{1}{2}}A\Tr{B}- \textstyle{\frac{1}{2}}B\Tr{A}+\textstyle{\frac{1}{2}}(\Tr{A}\Tr{B}-\Tr(A\circ B)).$$
From this a simple calculation shows that if $A=A_0$ then the condition \eqref{eq:crossprodcondition} 
is satisfied\footnote{This is the action of $\Sp_6$}.  We may therefore assume that $A_0=0$.  We find that
 $$ A\times \(\begin{smallmatrix} 0 & z & 0\\ -z & 0 & x \\ 0 & -x & 0 \end{smallmatrix}\) = \frac{1}{2} \(\begin{array}{ccc} -\Tr(z_1z) & y_1x & zx_1 +z_1x \\ xy_1 & 0 & zy_1 \\ xz_1+x_1z & y_1z & -\Tr(xx_1) \end{array}\). $$
Note that the fractor of $1/2$ can be safely ignored since it can be absorbed into $x_1,y_1,z_1$.

This imposes several conditions:
\begin{itemize}
 \item[(A)] $\Tr(z_1z)=\Tr(xx_1)=0$
 \item[(B)] $xy_1,y_1z\in \Oct^0$
 \item[(C)] $xz_1+x_1z\in \Oct^0$
\end{itemize}

Notice that for $w=a_1 s_1 + a_2 s_2 + a_3 s_3 + b_1 t_1 + b_2 t_2 + b_3 t_3 + c (s_4-t_4)\in \Oct^0$, we have
 $$ wz = wt_2 = a_2s_4+b_1s_3-b_3s_1-ct_2, $$
 $$ xw = s_1w = -a_2t_3+a_3t_2+b_1s_4-cs_1. $$
Combining this calculation with condition (A) shows that $a_2^z=b_1^x=0$.  With (B) it implies that $a_2^y=b_1^y=0$, and with condition (C) we get that $b_1^z=-a_2^x$.  Putting this all together, we have
 $$ x' = - y_1z = b_3^ys_1+c^yt_2, $$
 $$ y' = xz_1+x_1z = (a_3^z-c^x)t_2-(c^z+b_3^y)s_1, $$
 $$ z' = -xy_1 = -a_3^yt_2+c^ys_1. $$
This proves the Lemma.
\end{proof}

Proposition \ref{prop:E8Omega0} implies that $C_c^{\infty} (\Omega_2)$ is a submodule of 
$C_c^{\infty} (\Omega_0)$ and $C_c^{\infty} (\Omega_1)$ is a quotient.  Let $S_1$ and $S_2$ be the  
stabilizers of $((s_1,0,0),(0,0,0))$ and $((s_1,t_2,0),(0,0,0))$ in $G_2\times \GSp_6$, respectively. Let $P_m$ and $Q_m$, $m=1,2$, 
be the maximal parabolic subgroups in $G_2$ and $\GSp_6$, respectively, as introduced in 
\ref{g2_parabolics}  and \ref{sp_parabolic}. In particular, these parabolic 
groups come with maps $P_m \times Q_m \rightarrow \GL_m \times \GL_m$. Then $S_m$ is the inverse image of 
$\Delta(\GL(m))$, the diagonally embedded $\GL_m$ into $\GL_m \times \GL_m$. Now one can easily deduce that 
$\bfV_{\bar U}$, as a representation of  $G_2\times \GSp_6$,  has the following three sub quotients, here induction is 
not normalized.

\begin{enumerate} 
\item $C_c^{\infty} (\Omega_2)\cong \Ind^{G_2\times \GSp_6}_{P_2\times Q_2}(C_c^\infty(\GL_2))\otimes |i|^5$.
\item   $C_c^{\infty} (\Omega_1)\cong\Ind^{G_2\times \GSp_6}_{P_1\times Q_1}(C_c^\infty(\GL_1))\otimes|i|^{5}$. 
\item $\bfV_{\bar N}\cong \bfV(M)\otimes |i|^3  \oplus 1\otimes |i|^5 $.
\end{enumerate} 

Proposition \ref{prop:swE8} is simply a normalized version of this result.

\bibliographystyle{plain}
\nocite{sw}
\bibliography{matchbib}
\end{document}